\documentclass{amsart}
\usepackage{amsmath,amssymb,amsthm}
\newtheorem{thm}{Theorem}[section]

\newtheorem*{conj1}{Extended Lindel\"{o}f Hypothesis}

\newtheorem{lem}[thm]{Lemma}
\newtheorem*{thm*}{Theorem}

\makeatletter

\@addtoreset{equation}{section}
\makeatother
\begin{document}

\title{The distribution of lattice points with relatively $r$-prime}
\author{Wataru Takeda}
\address{Department of Mathematics,Kyoto University, Kitashirakawa Oiwake-cho, Sakyo-ku, Kyoto 606-8502,
Japan}
\email{takeda-w@math.kyoto-u.ac.jp}

\begin{abstract}
The distribution of lattice points with relatively $r$-prime is related to problems in the Number Theory such as the Extended Lindel\"{o}f Hypothesis and the Gauss Circle Problem. It is known that Sittinger's result is improved on the assumption of the Extended Lindel\"{o}f Hypothesis. In this paper, we improve Sittinger's result without assuming the Extended Lindel\"{o}f hypothesis. 
\end{abstract}
\maketitle
\section{Introduction}
The Geometry of Number was used to consider problems of the Number Theory by Minkowski. K. Rogers and H. P. F. Swinnerton-Dyer extended the Geometry of Number over number fields. Let $K$ be a number field and let $\mathcal{O}_K$ be its ring of integers. We consider an ordered $m$-tuple of ideals $(\mathfrak{a}_1, \mathfrak{a}_2,\ldots,\mathfrak{a}_m)$ of $\mathcal{O}_K$ as a lattice point in $K^m$. We say that a lattice point $(\mathfrak{a}_1, \mathfrak{a}_2,\ldots,\mathfrak{a}_m)$ is relatively $r$-prime, if there exists no prime ideal $\mathfrak{p}$ such that $\mathfrak{a}_1, \mathfrak{a}_2,\ldots,\mathfrak{a}_m\subset \mathfrak{p}^r$. 

In the case $K=\mathbf{Q}$, S. J. Benkoski proved that the density of the set of $m$-tuples of integers which are relatively $r$-prime is $1/\zeta(rm)$ in 1976 \cite{Be76}. And in general case, B. D. Sittinger proved the number of lattice points $(\mathfrak{a}_1, \mathfrak{a}_2,\ldots,\mathfrak{a}_m)$ in $K^m$ with relatively $r$-prime and $\mathfrak{Na_i}\le x$ for all $i=1,\ldots,m$ is \[\frac{c^m}{\zeta_K(rm)}x^m+\text{(Error term)},\]
where $\zeta_K$ is the Dedekind zeta function over $K$ and $c$ is a constant real number depending only on $K$ \cite{St10}.

Let $V_m^r(x,K)$ denote the number of lattice points $(\mathfrak{a}_1, \mathfrak{a}_2,\ldots,\mathfrak{a}_m)$ with relatively $r$-prime and $\mathfrak{Na_i}\le x$ for all $i=1,\ldots,m$ and let $E_m^r(x,K)$ denote its error term, i.e. $E_m^r(x,K)=V_m^r(x,K)-(cx)^m/\zeta_K(rm)$. 

We considered better upper bound of $E_m^r(x,K)$ on the assumption of the Extended Lindel\"{o}f Hypothesis in \cite{Ta16}. In this paper, we use some classical results to improve Sittinger's estimation without assuming the Extended Lindel\"{o}f hypothesis. Our main theorem is the following results.

\begin{thm*}
Let $\alpha(n)$ and $\beta(n)$ be constants \[\alpha(n)=\left\{
\begin{array}{ll}
\frac2n-\frac8{n(5n+2)}&\text{ if } 3\le n\le6,\\
\frac2n-\frac3{2n^2}&\text{ if } 7\le n\le9\rule[-2mm]{0mm}{6mm},\\
\frac3{n+6}-\varepsilon&\text{ if } n\ge10,
\end{array}
\right.
\text{ and \ } \beta(n)=\left\{
\begin{array}{ll}
\frac{10}{5n+2}&\text{ if } 3\le n\le6,\\
\frac 2n&\text{ if } 7\le n\le9\rule[-2mm]{0mm}{6mm},\\
0&\text{ if } n\ge10.
\end{array}
\right.\]

When $n=[K:\mathbf{Q}]$, then we get \[E_m^r(x,K)=\left\{
\begin{array}{ll}
O(x^{m-\alpha(n)}(\log x)^{\beta(n)})&\text{ if } rm\ge3,\\
O(x^{2-\alpha(n)}(\log x)^{2\beta(n)+1})&\text{ if } r=1 \text{ and } m=2,\\
O(x^{1-\alpha(n)/2}(\log x)^{2\beta(n)})&\text{ if } r=2 \text{ and } m=1,
\end{array}
\right.\]
for all $\varepsilon>0$.
\end{thm*}

As a result, we can improve Sittinger's result for $n\ge3$. Moreover, we get better results about $E_m^r(x,K)$ than assuming the Extended Lindel\"{o}f Hypothesis for all number field $K$ with $3=[K:\mathbf{Q}]$. 

\section{Dedekind zeta function over $K$}
The Dedekind zeta function $\zeta_K$ over $K$ is considered as a generalization of the Riemann zeta function and $\zeta_K$ is defined as 
\[\zeta_K(s)=\sum_{\mathfrak{a}}\frac1{\mathfrak{N}\mathfrak{a}^s},\]with the sum taken over all nonzero ideals of $\mathcal{O}_K$.

Lindel\"{o}f proposed that for all $\varepsilon>0$, we have $\zeta(1/2+it)=O(t^{\varepsilon})$ as $t\rightarrow\infty$ \cite{Li08}. This hypothesis is implied by the Riemann hypothesis, so this is very important hypothesis. The Extended Lindel\"{o}f Hypothesis is known as a generalization of the Lindel\"{o}f hypothesis.  The statement of this Hypothesis is written as follows.
\begin{conj1}For every $\varepsilon>0$, 
\[\zeta_K\left(\frac12+it\right)=O(t^\varepsilon)\ \text{ as }t\rightarrow\infty.\]
\end{conj1}
Too many mathematicians tried to get better estimate so far. D. R. Heath-Brown \cite{He88} proved that when $n=[K:\mathbf{Q}]$
\begin{equation}
\zeta_K\left(\frac12+it\right)=O(x^{n/6+\varepsilon})\ \text{ as }t\rightarrow\infty.\label{he}
\end{equation}
This result is known as the best result ever. We used the Extended Lindel\"{o}f Hypothesis to consider $I_K(x)$ in \cite{Ta16}. In this paper, we use one of results obtained from (\ref{he}) instead of  the Extended Lindel\"{o}f Hypothesis.

\section{The number of ideals}
In this section, we prepare for showing the main theorem. Let $I_K(x)$ be the number of ideals of $\mathcal{O}_K$ with $\mathfrak{Na_i}\le x$. We consider the value of $I_K(x)$ to estimate $E_m^r(x,K)$, because we know that 
\begin{equation}
V_m^r(x,K)=\sum_{\mathfrak{N}\mathfrak{a}\le x^{1/r}}\mu(\mathfrak{a})I_K\left(\frac x{\mathfrak{N}\mathfrak{a}^r}\right)^m.\label{vm}
\end{equation}
We will consider the sum (\ref{vm}), where $\mu(\mathfrak{a})$ is the M\"{o}bius function defined as 
\[\mu(\mathfrak{a})=\left\{
\begin{array}{ll}
1 &i\!f \ \mathfrak{a}=1,\\
(-1)^s &i\!f \ \mathfrak{a}=\mathfrak{p}_1\cdots \mathfrak{p}_s, \text{ where $\mathfrak{p}_1,\ldots, \mathfrak{p}_s$ are distinct prime ideals,}\\
0  &i\!f \  \mathfrak{a}\subset\mathfrak{p}^2 \text{ for some prime ideal } \mathfrak{p}.
\end{array}
\right.\] 

Considering with this fact, it is important to study the distribution of lattice points, i.e. that of ideals $I_K(x)$. In 1993, W. G. Nowak obtained the following result about $I_K(x)$ \cite{No93}.
\begin{lem}[cf.\cite{No93}]
\label{no}
When $n=[K:\mathbf{Q}]$, then we get
\[I_K(x)=cx+\left\{
\begin{array}{ll}
O(x^{1-\frac2n+\frac8{n(5n+2)}} (\log x)^{\frac{10}{5n+2}})& \text{ for } 3\le n\le6,\\
O(x^{1-\frac2n+\frac3{2n^2}} (\log x)^{\frac2n})& \text{ for } n\ge 7,
\end{array}
\right.\]
where \[c=\frac{2^{r_1}(2\pi)^{r_2}hR}{w\sqrt{|d_K|}},\]
and:

$h$ is the class number of $K$,

$r_1$ and $r_2$ is the number of real and complex absolute values of $K$ respectively,

$R$ is the regulator of $K$,

$w$ is the number of roots of unity in $\mathcal{O}^*_K$,

$d_K$ is discriminant of $K$.
\end{lem}

For all number field with $[K:\mathbf{Q}]\ge10$, Nowak's result is improved by H. Lao in 2010 \cite{La10}. Following result was shown  by using the estimate (\ref{he}).
\begin{lem}[cf. \cite{La10}]
\label{la}
When $n=[K:\mathbf{Q}]$, then we get
\[I_K(x)=c x+O(x^{1-\frac3{n+6}+\varepsilon})\]
for every $\varepsilon>0$,
\end{lem}

Combine above lemmas, following lemma about the distribution of ideals holds.  
\begin{lem}
\label{number}
When $n=[K:\mathbf{Q}]$, then we get
\[I_K(x)=cx+\left\{
\begin{array}{ll}
O(x^{1-\frac2n+\frac8{n(5n+2)}} (\log x)^{\frac{10}{5n+2}})& \text{ for } 3\le n\le6,\\
O(x^{1-\frac2n+\frac3{2n^2}} (\log x)^{\frac2n})& \text{ for } 7\le n\le9,\\
O(x^{1-\frac3{n+6}+\varepsilon})&\text{ for } n\ge10,
\end{array}
\right.\]
for all $\varepsilon>0$
\end{lem}
It is too elementary result that the number of lattice points $(\mathfrak{a}_1, \mathfrak{a}_2,\ldots,\mathfrak{a}_m)$ in $\mathbf{Q}^m$ with $\mathfrak{Na_i}\le x$ is $[x]^m$. On the other hand, we have no explicit results about the distribution of lattice points in $K^m$.
By using Lemma \ref{number}, we can estimate the number of lattice points $(\mathfrak{a}_1, \mathfrak{a}_2,\ldots,\mathfrak{a}_m)$ in $K^m$ with $\mathfrak{Na_i}\le x$.

\section{The proof of the main theorem}
In this section we will show the main theorem.  Using the equation (\ref{vm}) and Lemma \ref{number}, an approximation formula for $E_m^r(x,K)$ is obtained.

\begin{thm}
\label{main}
Let $\alpha(n)$ and $\beta(n)$ be constants \[\alpha(n)=\left\{
\begin{array}{ll}
\frac2n-\frac8{n(5n+2)}&\text{ if } 3\le n\le6,\\
\frac2n-\frac3{2n^2}&\text{ if } 7\le n\le9\rule[-2mm]{0mm}{6mm},\\
\frac3{n+6}-\varepsilon&\text{ if } n\ge10,
\end{array}
\right.
\text{ and \ } \beta(n)=\left\{
\begin{array}{ll}
\frac{10}{5n+2}&\text{ if } 3\le n\le6,\\
\frac 2n&\text{ if } 7\le n\le9\rule[-2mm]{0mm}{6mm},\\
0&\text{ if } n\ge10.
\end{array}
\right.\]

When $n=[K:\mathbf{Q}]$, then we get \[E_m^r(x,K)=\left\{
\begin{array}{ll}
O(x^{m-\alpha(n)}(\log x)^{\beta(n)})&\text{ if } rm\ge3,\\
O(x^{2-\alpha(n)}(\log x)^{2\beta(n)+1})&\text{ if } r=1 \text{ and } m=2,\\
O(x^{1-\alpha(n)/2}(\log x)^{2\beta(n)})&\text{ if } r=2 \text{ and } m=1,
\end{array}
\right.\]
for all $\varepsilon>0$.
\end{thm}

\begin{proof}As we remarked above, we know the equation (\ref{vm}); \[V_m^r(x,K)=\sum_{\mathfrak{N}\mathfrak{a}\le x^{1/r}}\mu(\mathfrak{a})I_K\left(\frac x{\mathfrak{N}\mathfrak{a}^r}\right)^m.\]
\begin{align*}
\intertext{First we use Lemma \ref{number} and the binomial theorem. Then we obtain}
V_m^r(x,K)&=\sum_{\mathfrak{N}\mathfrak{a}\le x^{1/r}}\mu(\mathfrak{a})\left(\frac {cx}{\mathfrak{N}\mathfrak{a}^r}+O\left(\left(\frac x{\mathfrak{N}\mathfrak{a}^r}\right)^{1-\alpha(n)}(\log x/\mathfrak{Na}^r)^{\beta(n)}\right)\right)^m\\
&=(cx)^m\sum_{\mathfrak{N}\mathfrak{a}\le x^{1/r}}\frac {\mu(\mathfrak{a})}{\mathfrak{N}\mathfrak{a}^{rm}}+O\left(\sum_{\mathfrak{N}\mathfrak{a}\le x^{1/r}}\left(\frac {x}{\mathfrak{N}\mathfrak{a}^r}\right)^{m-\alpha(n)}(\log x)^{\beta(n)}\right).\\
\intertext{By using the fact $\displaystyle{\sum_{\mathfrak{a}}\frac {\mu(\mathfrak{a})}{\mathfrak{N}\mathfrak{a}^{rm}}=\frac1{\zeta_K(rm)}}$, we get}
V_m^r(x,K)&=\frac{c^m}{\zeta_K(rm)}x^m-(cx)^m\sum_{\mathfrak{N}\mathfrak{a}> x^{1/r}}\frac {\mu(\mathfrak{a})}{\mathfrak{N}\mathfrak{a}^{rm}}+O\left(\sum_{\mathfrak{N}\mathfrak{a}\le x^{1/r}}\left(\frac {x}{\mathfrak{N}\mathfrak{a}^r}\right)^{m-\alpha(n)}(\log x)^{\beta(n)}\right).\\
\intertext{The first term $(cx)^m/\zeta_K(rm)$ is known as the principal term of $V_m^r(x,K)$, so $E_m^r(x,K)$ is left terms. Thus}
E_m^r(x,K)&=O\left(x^m\sum_{\mathfrak{N}\mathfrak{a}> x^{1/r}}\frac {\mu(\mathfrak{a})}{\mathfrak{N}\mathfrak{a}^{rm}}+\sum_{\mathfrak{N}\mathfrak{a}\le x^{1/r}}\left(\frac {x}{\mathfrak{N}\mathfrak{a}^r}\right)^{m-\alpha(n)}(\log x)^{\beta(n)}\right).\\
\intertext{Now we estimate how fast above first sum grows. From Theorem \ref{number} we can estimate $I_K(x)-I_K(x-1)=O\left(x^{1-\alpha(n)}(\log x)^{\beta(n)}\right)$, so we have}
\sum_{\mathfrak{N}\mathfrak{a}> x^{1/r}}\frac {\mu(\mathfrak{a})}{\mathfrak{N}\mathfrak{a}^{rm}}&=O\left(x^m\int_{x^{1/r}}^{\infty}\frac{y^{1-\alpha(n)}(\log y)^{\beta(n)}}{y^{rm}}\ dy\right)\\
&=O(x^{(2-\alpha(n))/r}(\log x)^{\beta(n)}).\\
\intertext{Next we estimate how fast above second sum grows. As well as first sum, $I_K(x)-I_K(x-1)=O\left(x^{1-\alpha(n)}(\log x)^{\beta(n)}\right)$ holds, so we have}
\sum_{\mathfrak{N}\mathfrak{a}\le x^{1/r}}\left(\frac {x}{\mathfrak{N}\mathfrak{a}^r}\right)^{m-1/2+\varepsilon}&=O\left(x^{m-\alpha(n)}\left(1+\int_1^{x^{1/r}}\frac{y^{1-\alpha(n)}(\log y)^{\beta(n)}}{y^{r(m-\alpha(n))}}(\log y)^{\beta(n)}\ dy\right)\right)\\
&=\left\{
\begin{array}{ll}
O(x^{m-\alpha(n)}(\log x)^{\beta(n)})&i\!f \ rm\ge3,\\
O(x^{2-\alpha(n)}(\log x)^{2\beta(n)+1})&i\!f\   r=1 \text{ and } m=2,\\
O(x^{1-\alpha(n)/2}(\log x)^{2\beta(n)})&i\!f \ r=2 \text{ and } m=1.
\end{array}
\right.\\
\intertext{Hence we get}
E_m^r(x,K)&=\left\{
\begin{array}{ll}
O(x^{m-\alpha(n)}(\log x)^{\beta(n)})&i\!f \ rm\ge3,\\
O(x^{2-\alpha(n)}(\log x)^{2\beta(n)+1})&i\!f \ r=1 \text{ and } m=2,\\
O(x^{1-\alpha(n)/2}(\log x)^{2\beta(n)})&i\!f \ r=2 \text{ and } m=1.
\end{array}
\right.
\intertext{This proves the main Theorem.}
\end{align*}
\end{proof}
In 2010, B. D. Sittinger showed the following theorem about lattice points with relatively $r$-prime over number field $K$.
\begin{thm}[cf. \cite{St10}]
When $n=[K:\mathbf{Q}]$

\[E_m^r(x,K)=\left\{
\begin{array}{ll}
O(x^{m-1/n}) & \text{ if } m\ge3, \text{ or } m=2 \text{ and } r\ge2,\\
O(x^{2-1/n}\log x)& \text{ if }m=2 \text{ and } r=1,\\
O(x^{1-1/n}\log x) & \text{ if }m=1 \text{ and }  \frac{n(r-2)}{r-1}=1,\\
O(x^{1-1/n})& \text{ if }m=1 \text{ and } \frac{n(r-2)}{r-1}>1,\\
O(x^{(2-1/n)/r})&\text{ if }m=1 \text{ and } \frac{n(r-2)}{r-1}<1.
\end{array}
\right.\]

\end{thm}
Considering the Sittinger's result \cite{St10}, we can improve the order of $E_m(x,K)$ for all number field $K$ with $[K:\mathbf{Q}]\ge3$. 

\section{appendix}
In this paper we considered about all number fields, but our results can be improved for all abelian extension $K/\mathbf{Q}$. It is well known that for an abelian extension $K/\mathbf{Q}$ with $[K:\mathbf{Q}]\ge4$,
\begin{equation}
I_K(x)=cx+O(x^{1-3/(n+2)}).\label{ti}
\end{equation}
Using this approximation (\ref{ti}), the following better result for abelian extension field $K$ is obtained in a way similar to the proof of Theorem \ref{main}.
\begin{thm}
For all abelian extension field $K$ with $[K:\mathbf{Q}]\ge4$, we get
 \[E_m^r(x,K)=\left\{
\begin{array}{ll}
O(x^{1-\frac3{2(n+2)}+\varepsilon})&\text{ if } r=2 \text{ and } m=1,\\
O(x^{m-\frac3{n+2}+\varepsilon})&\ otherwise.\\
\end{array}
\right.\]
\end{thm}
It goes without saying that the better approximation formula for $I_K(x)$ we have, the better results for $E_m^r(x,K)$ we get. On the contrary, if we can get the exact order of $E_m^r(x,K)$, the explicit approximation formula for $I_K(x)$ is obtained from considering the exact order of $E_m^r(x,K)$. And it is proven that there are some relation between the distribution of lattice points with relatively $r$-prime and problems in the Number Theory in \cite{Ta16}.  We think it is important that studying the exact order of $E_m^r(x,K)$ without using an approximation formula for $I_K(x)$.




\end{document}